\documentclass[12pt,reqno]{amsart}
\usepackage[margin=1in]{geometry}
\usepackage{amscd,amsfonts,amsmath,amssymb,amsthm,latexsym,mathrsfs,textcomp,verbatim}
\usepackage{accents}
\usepackage{bm}
\usepackage{bbm}
\usepackage{cite}
\usepackage{color}
\usepackage{constants}
\usepackage{csquotes}
\usepackage{enumerate}
\usepackage{etoolbox}
\usepackage{hypbmsec}
\usepackage[breaklinks=true,colorlinks=true,linkcolor=teal,citecolor=teal,filecolor=teal,urlcolor=black!50!blue]{hyperref}
\hypersetup{linktocpage}
\usepackage[hiresbb]{graphicx,xcolor}
\usepackage[capitalise]{cleveref}
\usepackage{mathtools}
\usepackage{soul}
\usepackage{tikz}
\usepackage{tikz-cd}
\usetikzlibrary{shapes.geometric}
\usetikzlibrary{shapes.misc}
\usetikzlibrary{positioning}
\allowdisplaybreaks[4]

\newtheorem{theorem}{Theorem}[section]
\newtheorem{corollary}[theorem]{Corollary}

\newtheorem{lemma}[theorem]{Lemma}

\theoremstyle{definition}

\theoremstyle{remark}
\newtheorem{remark}{Remark}

\numberwithin{equation}{section}

\renewcommand{\Re}{\mathrm{Re}}
\renewcommand{\tilde}{\widetilde}

\renewcommand{\C}{\mathbb{C}}

\newcommand{\N}{\mathbb{N}}

\newcommand{\R}{\mathbb{R}}
\newcommand{\Z}{\mathbb{Z}}
\newcommand{\GL}{\mathrm{GL}}
\newcommand{\SL}{\mathrm{SL}}

\renewcommand{\epsilon}{\varepsilon}

\patchcmd{\section}{\scshape}{\bfseries}{}{}
\makeatletter
\renewcommand{\@secnumfont}{\bfseries}
\makeatother

\makeatletter\newcommand{\tpmod}[1]{{\@displayfalse \pmod{#1}}}
\makeatletter
\@namedef{subjclassname@2020}{\textup{2020} Mathematics Subject Classification}
\makeatother

\begin{document}

\title{On Multiple Shifted Convolution Sums}

\author{Ikuya Kaneko}
\address{The Division of Physics, Mathematics and Astronomy, California Institute of Technology, 1200 E. California Blvd., Pasadena, CA 91125, USA}
\email{ikuyak@icloud.com}
\urladdr{\href{https://sites.google.com/view/ikuyakaneko/}{https://sites.google.com/view/ikuyakaneko/}}

\thanks{The author acknowledges the support of the Masason Foundation.}



\date{\today}

\begin{abstract}
We prove strong estimates for averages of shifted convolution sums consisting~of quadratic twists of $\GL_{2}$ $L$-functions. The key input involves the circle method together with standard tools such as Vorono\u{\i}, quadratic reciprocity, amplification, and divisor switching.
\end{abstract}

\maketitle
\tableofcontents

\section{Introduction}\label{introduction}

\subsection{Brief Retrospection}\label{brief-retrospection}
Given arithmetically interesting sequences of complex numbers $\{a(n) \}_{n \in \N}$ and $\{b(n) \}_{n \in \N}$, the \emph{shifted convolution problem} or the \emph{generalised additive divisor problem} asks for determining the behaviour of (or even just detecting nontrivial cancellations in) correlations of the shape
\begin{equation}\label{eq:ab}
\sum_{T \leq n \leq 2T} a(n) b(n+h).
\end{equation}
Such a sum pertains to various arithmetic problems depending on the sequences $a(n)$ and~$b(n)$. Achieving subconvex bounds for~\eqref{eq:ab} yields salient and sometimes unexpected applications. The archetype is when $a(n)$ and $b(n)$ come from the von Mangoldt function, M\"{o}bius function, or~the divisor function, in which case~\eqref{eq:ab} is related to the Hardy--Littlewood~prime $k$-tuple conjecture~\cite{HardyLittlewood1923-2}, Chowla conjecture~\cite{Chowla1965}, gaps between multiplicative sequences~\cite{Hooley1971,Hooley1994}, and moments of $L$-functions~\cite{ConreyKeating2015}, to name a few. Another example is when $a(n)$ and $b(n)$ come from $\GL_{2}$ Hecke eigenvalues, in which case~\eqref{eq:ab} is related to the subconvexity problem and quantum unique ergodicity. For further details, see~\cite{BlomerHarcos2008,Blomer2004,DeshouillersIwaniec1982,DukeFriedlanderIwaniec1993,Harcos2003,Holowinsky2009,Holowinsky2010,KowalskiMichelVanderKam2002,Leung2022,Leung2022-2,Maga2018,Michel2004,Michel2022,Topacogullari2016,Topacogullari2017,Topacogullari2018}.

\subsection{Statement of the Main Result}\label{statement-of-the-main-result}
It is often beneficial for applications to consider~\eqref{eq:ab} with an averaging over the shifts $h$ in a dyadic interval $[H, 2H]$. Fix~a Hecke--Maa{\ss} cusp~form on the modular surface $\SL_{2}(\Z) \backslash \mathbb{H}$, where $\mathbb{H} \coloneqq \{z = x+iy \in \C: y > 0 \}$ is the upper half-plane upon which the modular group acts via M\"{o}bius transformations. Given a fundamental discriminant $d$, let $\chi_{d} = \left(\frac{d}{\cdot} \right)$ be the primitive quadratic character modulo $|d|$. Then $\varphi \otimes \chi_{d}$ boils down to a Hecke--Maa{\ss} newform of level $|d|^{2}$ and principal nebentypus whose $L$-function is expressed in terms~of a Dirichlet series and an Euler product, both converging absolutely for $\Re(s) > 1$:
\begin{equation*}
L(s, \varphi \otimes \chi_{d}) \coloneqq \sum_{n = 1}^{\infty} \frac{\lambda_{\varphi}(n) \chi_{d}(n)}{n^{s}} = \prod_{p} \left(1-\frac{\lambda_{\varphi}(p) \chi_{d}(p)}{p^{s}}+\frac{\chi_{d}(p)^{2}}{p^{2s}} \right)^{-1}.
\end{equation*}

For $1 \leq H \leq T$, we define the \emph{multiple shifted convolution problem}\footnote{This name stems from multiple $L$-functions, namely $L$-functions whose coefficients are again $L$-functions. They have proven to be a quite powerful and elegant tool that in some cases is capable of yielding results~that are not yet available with other techniques. To circumvent terminological redundancy, it is convenient in this paper to call the averaged version~\eqref{eq:def-M} a (multiple) shifted convolution problem, albeit being less standard.} by
\begin{equation}\label{eq:def-M}
\mathcal{M}_{\varphi}(T, H) \coloneqq \sideset{}{^{\ast}} \sum_{H \leq h \leq 2H} \ \sideset{}{^{\ast}} \sum_{T \leq n \leq 2T} 
L \left(\frac{1}{2}, \varphi \otimes \chi_{8n} \right) L \left(\frac{1}{2}, \varphi \otimes \chi_{8(n+h)} \right),
\end{equation}
where the asterisks mean that each sum runs through positive squarefree integers $n$ and~$n+h$ such that $(n, 2) = 1$ and $(n+h, 2) = 1$, respectively. In analogy with the shifted convolution problem for $\GL_{2}$, one should expect substantial cancellations in $\mathcal{M}_{\varphi}(T, H)$. In this paper,~we study an unconditional quantitative manifestation of this conjecture in certain ranges of $H$.
\begin{theorem}\label{main}
Let $\varphi$ be a Hecke--Maa{\ss} cusp form on $\SL_{2}(\Z) \backslash \mathbb{H}$. Then we have for any~$\epsilon > 0$ that
\begin{equation*}
\mathcal{M}_{\varphi}(T, H) \ll_{\varphi, \epsilon} T^{\frac{5}{4}+\epsilon}, \qquad T^{\frac{1}{4}} \leq H \leq \sqrt{T}.
\end{equation*}
\end{theorem}

In down-to-earth terms, Theorem~\ref{main} asserts that the total saving that we attain is roughly of size $H T^{-\frac{1}{4}} \geq 1$, since the trivial bound is $O_{\varphi, \epsilon}(HT^{1+\epsilon})$ via the second moment bound for quadratic twists and Cauchy--Schwarz (trivially bounding the $h$-sum). Theorem~\ref{main}, however, falls well shy of the truth since one would expect the best possible bound to be $O_{\varphi, \epsilon}(T^{1+\epsilon})$.

\begin{remark}
Our method also works when $\varphi$ is either holomorphic or Eisenstein, but we here restrict to the Maa{\ss} case for brevity, which is fundamentally formidable over the~others in the sense that the Ramanujan--Petersson conjecture for Maa{\ss} forms is unproven up until~now.
\end{remark}

\begin{remark}
For brevity, we restrict to positive fundamental discriminants of the form $8n$ and $8(n+h)$, but we may deal similarly with all discriminants. This assumption is also imposed~in the work of Soundararajan--Young~\cite{SoundararajanYoung2010} and Li~\cite{Li2022} and enables subsequent~discussions.
\end{remark}

The proof of Theorem~\ref{main} relies on the Duke--Friedlander--Iwaniec \emph{circle method} along~with standard manipulations including Vorono\u{\i}, Poisson, orthogonality, and quadratic reciprocity. The crucial ingredients include \emph{divisor switching}, which guarantees a conductor drop in~other summations. Nonetheless, this manoeuvre sacrifices the complementary divisor being larger than the original divisor. To eschew this drawback, we utilise an \emph{amplification}. It appears at first glance that applying the lengthening here is nonsense, but in fact facilitates a conductor drop in Poisson summation. This should be thought of as an analogue of the trick of Li~\cite{Li2022}. It behoves us to mention that the proof of Theorem~\ref{main} does not require Cauchy--Schwarz because the sum over $n$ in~\eqref{eq:def-M} becomes symmetrical after using the circle method.

We comment that the multiple shifted convolution problem that we address pertains to~the quantum unique ergodicity conjecture for half-integral weight Eisenstein series. In fact, their $n$-th Fourier coefficient $c_{t}(n)$ may be written in the form
\begin{equation*}
c_{t}(n) = \frac{(\star)}{\zeta(1+it)} \cdot L \left(\frac{1}{2}+it, \chi_{n} \right),
\end{equation*}
where $(\star)$ hides some fairly tame fudge factors. The contribution of an incomplete Eisenstein series boils down to the second moment problem for $L(\frac{1}{2}+it, \chi_{n})$, while the contribution of an incomplete Poincar\'{e} series boils down to the shifted convolution problem of the shape
\begin{equation*}
\sideset{}{^{\ast}} \sum_{n \sim t} L \left(\frac{1}{2}+it, \chi_{n} \right) L \left(\frac{1}{2}+it, \chi_{n+h} \right)
\end{equation*}
for any fixed $h \ne 0$. Choosing $\varphi = |\cdot|_{\mathbb{A}}^{it}$ for $t \asymp T$ in~\eqref{eq:def-M} recovers the above expression (but with an averaging over the shifts $h$). Note that one would expect
\begin{equation*}
\sideset{}{^{\ast}} \sum_{n \sim t} \left|L \left(\frac{1}{2}+it, \chi_{n} \right) L \left(\frac{1}{2}+it, \chi_{n+h} \right) \right| 
\asymp t \sqrt{\log t}.
\end{equation*}
This type of bound appears in the work of Holowinsky--Soundararajan~\cite{Holowinsky2009,HolowinskySoundararajan2010,Soundararajan2010-2}, which adopts Shiu's bound; see~\cite{ElliottMorenoShahidi1984,Nair1992,NairTenenbaum1998,Shiu1980}. In the half-integral weight case, Shiu's bound does not work as the coefficients are not multiplicative, but a tight upper bound follows from the Maa{\ss}--Selberg relation instead. Petridis--Raulf--Risager~\cite{PetridisRaulfRisager2014}~established quantum unique ergodicity for half-integral weight Eisenstein series under subconvex bounds for multiple Dirichlet series. For a general theory of multiple Dirichlet series and applications thereof, see for instance~\cite{Blomer2011,BlomerGoldmakherLouvel2014,Bump,BumpFriedbergGoldfeld2012,BumpFriedbergGoldfeldHoffstein2006,BumpFriedbergHoffstein1996,Cech2022-2,Cech2022,Cech2023,ChintaGunnells2007,ChintaGunnells2010,Dahl2015,Dahl2018,DiaconuGoldfeldHoffstein2003,FriedbergHoffsteinLieman2003,GaoZhao2023,GoldfeldHoffstein1985,PetridisRaulfRisager2014,Sawin2023,Wachter2021}.

\subsection{Discussions on the Proof}\label{discussions-on-the-proof}
This section unveils a heuristic argument for Theorem~\ref{main} in a back-of-the-envelope fashion, giving a high-level sketch geared to experts. It is structured such that any reader can understand the flow of the discussion. There is a caveat that we here ignore various technicalities such as complicated smooth weights and a number of coprimality conditions and common divisors. We pretend that everything is coprime to everything, which is morally not too far from reality. Furthermore, we have freedom to use~quadratic~reciprocity, which allows us to flip the numerator and denominator in the Jacobi--Kronecker symbol up~to a correction factor that we shall elide. Given a Hecke--Maa{\ss} cusp form $\varphi$ and $T^{\frac{1}{4}} \leq H \leq \sqrt{T}$, we wish to estimate nontrivially a multiple shifted convolution problem roughly of the shape\footnote{Here and henceforth, the meaning of the symbol $\approx$ is left vague on purpose. Furthermore, we shall write temporarily $n \sim T$ in place of $T \leq n \leq 2T$, which applies to other summations.}
\begin{equation*}
\mathcal{M}_{\varphi}(T, H) \approx \sum_{h \sim H} \sum_{n \sim T} 
L \left(\frac{1}{2}, \varphi \otimes \chi_{n} \right)  L \left(\frac{1}{2}, \varphi \otimes \chi_{n+h} \right),
\end{equation*}
where we drop the superscripts $\ast$ in the definition~\eqref{eq:def-M} for simplicity. While such individual shifted convolution sums are out of reach of current technology, we can leverage an averaging over $h$ for a gain. We now insert the Kronecker symbol $\delta(m = n)$ to separate the oscillations trapped in $\mathcal{M}_{\varphi}(T, H)$, so that the Duke--Friedlander--Iwaniec circle method implies
\begin{equation*}
\mathcal{M}_{\varphi}(T, H) \approx \frac{1}{T} \sum_{c \sim \sqrt{T}} \ \sideset{}{^{\ast}} \sum_{a \tpmod{c}} 
\sum_{h \sim H} e \left(-\frac{ah}{c} \right) \left|\sum_{m \sim T} 
L \left(\frac{1}{2}, \varphi \otimes \chi_{m} \right) e \left(\frac{am}{c} \right) \right|^{2}.
\end{equation*}
By Poisson summation, the sum over $h$ transforms into
\begin{equation*}
\sum_{h \sim H} e \left(-\frac{ah}{c} \right) 
\approx H \sum_{h \sim \frac{\sqrt{T}}{H}} \delta(h \equiv a \tpmod{c}).
\end{equation*}
while the sum over $n$ transforms into (via the approximate functional equation)
\begin{equation*}
\sum_{m \sim T} L \left(\frac{1}{2}, \varphi \otimes \chi_{m} \right) e \left(\frac{am}{c} \right) 
\approx \sum_{\ell \sim T} \sum_{m \sim \sqrt{T}} \lambda_{\varphi}(\ell) \left(\frac{cm}{\ell} \right) \delta(m \equiv a \ell \tpmod{c}).
\end{equation*}
Therefore, we obtain something roughly of the shape
\begin{equation*}
\mathcal{M}_{\varphi}(T, H) \approx \frac{H}{T} \sum_{c \sim \sqrt{T}} \sum_{h \sim \frac{\sqrt{T}}{H}} 
\left|\sum_{\ell \sim T} \sum_{m \sim \sqrt{T}} \lambda_{\varphi}(\ell) \left(\frac{cm}{\ell} \right) \delta(m \equiv h \ell \tpmod{c}) \right|^{2}.
\end{equation*}
The square-root cancellation heuristic implies that the best possible bound for the right-hand side is $O_{\varphi, \epsilon}(T^{1+\epsilon})$. For the ensuing analysis, it is now convenient to introduce an~amplification parameter $1 \leq L \leq \sqrt{T}$ and elongate the sum over $c$ by $L$. Opening the square, the problem boils down to determining bounds for
\begin{equation*}
\mathcal{M}_{\varphi}(T, H) \ll \frac{1}{\sqrt{T}} \sum_{c \sim L \sqrt{T}} \sum_{\ell_{1}, \ell_{2} \sim T} \sum_{m, n \sim \sqrt{T}} \lambda_{\varphi}(\ell_{1}) \lambda_{\varphi}(\ell_{2}) \left(\frac{cm}{\ell_{1}} \right) \left(\frac{cn}{\ell_{2}} \right) \delta(\ell_{1} n \equiv \ell_{2} m \tpmod{c}).
\end{equation*}
Divisor switching then comes into play, and we write
\begin{equation*}
\ell_{1} n = \ell_{2} m+cq, \qquad c \sim L \sqrt{T}, \qquad q \sim \frac{T}{L}.
\end{equation*}
It replaces a congruence condition modulo $c$ with a congruence condition modulo $q$, achieving a huge conductor drop simultaneously in the other variables. Without an amplification, the complementary divisor $q$ would be much larger than the initial divisor $c$. Hence, there holds
\begin{equation*}
\mathcal{M}_{\varphi}(T, H) \ll \frac{1}{\sqrt{T}} \sum_{q \sim \frac{T}{L}} \sum_{\ell_{1}, \ell_{2} \sim T} \sum_{m, n \sim \sqrt{T}} \lambda_{\varphi}(\ell_{1}) \lambda_{\varphi}(\ell_{2}) \left(\frac{q}{\ell_{1} \ell_{2}} \right) \delta(\ell_{1} n \equiv \ell_{2} m \tpmod{q}).
\end{equation*}
By Poisson summation, the sum over $m$ transforms into
\begin{equation*}
\sum_{m \sim \sqrt{T}} \delta(\ell_{1} n \equiv \ell_{2} m \tpmod{q}) 
\approx \frac{L}{\sqrt{T}} \sum_{m \sim \frac{\sqrt{T}}{L}} e \left(\frac{\ell_{1} \overline{\ell_{2}} mn}{q} \right),
\end{equation*}
while the sum over $n$ transforms into
\begin{equation*}
\sum_{n \sim \sqrt{T}} e \left(\frac{\ell_{1} \overline{\ell_{2}} mn}{q} \right) 
\approx \sqrt{T} \sum_{n \sim \frac{\sqrt{T}}{L}} \delta(\ell_{1} m \equiv \ell_{2} n \tpmod{q}).
\end{equation*}
By orthogonality, one expands
\begin{equation*}
\delta(\ell_{1} m \equiv \ell_{2} n \tpmod{q}) \approx \frac{1}{q} \ \sideset{}{^{\star}} \sum_{b \tpmod{q}} e \left(\frac{b(\ell_{1} m-\ell_{2} n)}{q} \right),
\end{equation*}
where $\star$ denotes summation restricted to reduced residue classes. To handle the sums over~$\ell_{1}$ and $\ell_{2}$, note that~\cite[Proposition 3.8 (iii)]{JacquetLanglands1970} or~\cite[Theorem 3.1 (ii)]{AtkinLi1978} implies that there exists a Hecke--Maa{\ss} newform $\varphi \otimes (\frac{q}{\cdot})$ of level $q^{2}$ and trivial nebentypus such that~$\lambda_{\varphi \otimes (\frac{q}{\cdot})}(\ell) = \lambda_{\varphi}(\ell) (\frac{q}{\ell})$. Hence, by $\GL_{2}$ Vorono\u{\i} summation, the sum over $\ell_{1}$ transforms~into
\begin{equation*}
\sum_{\ell_{1} \sim T} \lambda_{\varphi \otimes (\frac{q}{\cdot})}(\ell_{1}) e \left(\frac{b \ell_{1} m}{q} \right) 
\approx L \sum_{\ell_{1} \sim \frac{T}{L^{2}}} 
\lambda_{\varphi \otimes (\frac{q}{\cdot})}(\ell_{1}) e \left(\frac{\overline{b} \ell_{1} \overline{m}}{q} \right),
\end{equation*}
while the sum over $\ell_{2}$ transforms into
\begin{equation*}
\sum_{\ell_{2} \sim T} \lambda_{\varphi \otimes (\frac{q}{\cdot})}(\ell_{2}) e \left(-\frac{b \ell_{2} n}{q} \right) 
\approx L \sum_{\ell_{2} \sim \frac{T}{L^{2}}} 
\lambda_{\varphi \otimes (\frac{q}{\cdot})}(\ell_{2}) e \left(-\frac{\overline{b} \ell_{2} \overline{n}}{q} \right),
\end{equation*}
Summing over $b \tpmod{q}$ via orthogonality yields
\begin{equation*}
\mathcal{M}_{\varphi}(T, H) \ll \frac{L^{3}}{\sqrt{T}} \sum_{q \sim \frac{T}{L}} \sum_{\ell_{1}, \ell_{2} \sim \frac{T}{L^{2}}} \sum_{m, n \sim \frac{\sqrt{T}}{L}} \lambda_{\varphi \otimes (\frac{q}{\cdot})}(\ell_{1}) \lambda_{\varphi \otimes (\frac{q}{\cdot})}(\ell_{2}) \delta(\ell_{1} n \equiv \ell_{2} m \tpmod{q}).
\end{equation*}
As an endgame, we employ the Rankin--Selberg bound for the Hecke eigenvalues and estimate everything trivially, deducing
\begin{equation*}
\mathcal{M}_{\varphi}(T, H) \ll_{\varphi, \epsilon} L^{-\frac{7}{2}} T^{3+\epsilon} = T^{\frac{5}{4}+\epsilon},
\end{equation*}
where we optimise $L = \sqrt{T}$. This finishes the sketch of the proof of Theorem~\ref{main}.


\subsection{A Road Map and Notation}
Sections~\ref{arithmetic-machinery} and~\ref{automorphic-machinery} assemble requisite tools for the proof of Theorem~\ref{main}. In Section~\ref{proof}, we prove Theorem~\ref{main} along the same lines as in Section~\ref{discussions-on-the-proof}.

Throughout the paper, we make constant use of the notation $e(x) = e^{2\pi ix}$. We use $\varepsilon > 0$ to denote an arbitrarily small positive quantity that is possibly different in each instance. The Vinogradov symbol $f \ll_{\nu} g$ or the big $O$ notation $f = O_{\nu}(g)$ indicates that there exists an effectively computable constant $c_{\nu} > 0$, depending at most on $\nu$, such that $|f(z)| \leq c_{\nu} |g(z)|$ for all $z$ in a specified range. If no parameter $\nu$ is present, then $c$ is absolute. The Kronecker symbol $\delta(\mathrm{S})$ detects $1$ or $0$ according as the statement $\mathrm{S}$ is true or not.

\subsection*{Acknowledgements}
The author is indebted to Wing Hong Leung for helpful comments.

\section{Arithmetic Toolbox}\label{arithmetic-machinery}
This section compiles the arithmetic machinery that we shall need later. In particular, we formulate a version of the circle method (due to Duke--Friedlander--Iwaniec) and the~Poisson summation formula. Some fundamental properties of quadratic characters are also presented.

\subsection{$\delta$-Symbols}
There are two oscillations contributing to the shifted convolution problem that we address. The idea is to separate these oscillations via the circle method or the delta method. One seeks for a Fourier expansion that matches the Kronecker symbol $\delta(n = 0)$.
\begin{lemma}[Leung~\cite{Leung2021,Leung2022}]\label{DeltaCor}
Let $n \in \Z$ be such that $|n| \ll N$, $q \in \N$, and let~$C > N^{\epsilon}$. Let $U \in C_{c}^{\infty}(\R)$ and $W \in C_{c}^{\infty}([-2, -1] \cup [1, 2])$ be nonnegative even functions such that $U(x) = 1$ for $x \in [-2, 2]$. Then we have that
\begin{equation*}
\delta(n = 0) = \frac{1}{\mathcal{C}} \sum_{c = 1}^{\infty} \frac{1}{cq} 
\sum_{a \tpmod{cq}} e \left(\frac{an}{cq} \right) V_{0} \left(\frac{c}{C}, \frac{n}{cCq} \right),
\end{equation*}
where
\begin{equation*}
\mathcal{C} \coloneqq \sum_{c = 1}^{\infty} W \left(\frac{c}{C} \right) \sim C,
\end{equation*}
and
\begin{equation*}
V_{0}(x, y) \coloneqq W(x) U(x) U(y)-W(y) U(x) U(y)
\end{equation*}
is a smooth function satisfying $V_{0}(x, y) \ll \delta(|x|, |y| \ll 1)$.
\end{lemma}

By~\cite[Theorem 1]{HeathBrown1996}, Lemma~\ref{DeltaCor} is equivalent to the Duke--Friedlander--Iwaniec~\cite{DukeFriedlanderIwaniec1994} circle method with a simpler weight function $V_{0}$ that constrains $|n| \ll cCQ$. This particular feature is beneficial in the proof of Theorem~\ref{main}. See~\cite{KanekoLeung2023,Leung2021,Leung2022} for further details.

\subsection{Poisson Summation}\label{Poisson-summation}
For $n \in \mathbb{N}$ and an integrable function $w \colon \mathbb{R}^{n} \to \mathbb{C}$, denote its Fourier transform by
\begin{equation*}
\widehat{w}(y) \coloneqq \int_{\R^{n}} w(x) e(-\langle x, y \rangle) dx,
\end{equation*}
where $\langle \cdot, \cdot \rangle$ stands for the standard inner product on $\mathbb{R}^{n}$. Moreover, if $c \in \N$ and $K \colon \mathbb{Z} \to \mathbb{C}$ is a periodic function of period $c$, then its Fourier transform $\widehat{K}$ is again the periodic function of period $c$:
\begin{equation*}
\widehat{K}(n) \coloneqq \sum_{a \tpmod{c}} K(a) e \left(-\frac{an}{c} \right).
\end{equation*}
Note that there is a minor inconsistency in sign choices, namely $\widehat{\widehat{K}}(n) = K(-n)$ for all $n \in \Z$.

We invoke a form of the Poisson summation formula with a $c$-periodic function involved.
\begin{lemma}[{Fouvry--Kowalski--Michel~\cite[Lemma~2.1]{FouvryKowalskiMichel2015}}]\label{Fouvry-Kowalski-Michel}
For any $c \in \mathbb{N}$, any $c$-periodic function $K$, and any even smooth function $V$ compactly supported on $\mathbb{R}$, we have that
\begin{equation*}
\sum_{n = 1}^{\infty} K(n) V(n) = \frac{1}{c} \sum_{n \in \mathbb{Z}} \widehat{K}(n) \widehat{V} \left(\frac{n}{c} \right).
\end{equation*}
\end{lemma}

\subsection{Quadratic Characters}
We adhere to the notation of~\cite{Blomer2011,DiaconuGoldfeldHoffstein2003}. Let $d$ and~$n$~be odd positive integers that we factorise uniquely as $d = d_{0} d_{1}^{2}$ with $d_{0}$ squarefree and $n = n_{0} n_{1}^{2}$ with $n_{0}$ squarefree. Define the Jacobi--Kronecker symbol by
\begin{equation*}
\left(\frac{d}{n} \right) \coloneqq \prod_{p^{v} \parallel n} \left(\frac{d}{p} \right)^{v},
\end{equation*}
where for an odd prime $p$, we denote by $(\frac{d}{p})$ the standard Legendre symbol. Then the symbol $(\frac{d}{n})$ is extended to all odd $n \in \Z$ (cf.~\cite[p.442]{Shimura1973} and~\cite[p.147, 187--188]{Koblitz1984}). We write
\begin{equation*}
\chi_{d}(n) \coloneqq \left(\frac{d}{n} \right) \eqqcolon \tilde{\chi}_{n}(d).
\end{equation*}
The character $\chi_{d}$ is the Jacobi--Kronecker symbol of conductor $d_{0}$ if $d \equiv 1 \tpmod{4}$ and $4d_{0}$ if $d \equiv 3 \tpmod{4}$. By definition, we know
\begin{equation*}
\chi_{d}(2) = 
    \begin{cases}
    1 & \text{if $d \equiv 1 \tpmod{8}$},\\
    -1 & \text{if $d \equiv 5 \tpmod{8}$},\\
    0 & \text{if $d \equiv 3 \tpmod{4}$},
    \end{cases}
\end{equation*}
and $\chi_{d}(-1) = 1$, namely $\chi_{d}$ is even. Quadratic reciprocity~\cite[Theorem~3.5]{IwaniecKowalski2004} states that for relatively prime odd positive integers $d$ and $n$,
\begin{equation}\label{eq:quadratic-reciprocity}
\bigg(\frac{d}{n} \bigg) \bigg(\frac{n}{d} \bigg) = (-1)^{\frac{(d-1)(n-1)}{4}}.
\end{equation}
This implies in particular that
\begin{equation*}
\tilde{\chi}_{n} = 
    \begin{cases}
    \chi_{n} & \text{if $n \equiv 1 \tpmod{4}$},\\
    \chi_{-n} & \text{if $n \equiv 3 \tpmod{4}$}.
    \end{cases}
\end{equation*}

\subsection{Gau{\ss} Sums}
For a Dirichlet character $\chi \tpmod{c}$, orthogonality asserts
\begin{equation}\label{orthogonality-of-Dirichlet-characters}
\sum_{a \tpmod c} \chi(a) = 
	\begin{cases}
	\varphi(c) & \text{if $\chi = \chi_{0}$},\\
	0 & \text{otherwise},
	\end{cases}
\qquad \sum_{\chi \tpmod c} \chi(a) = 
	\begin{cases}
	\varphi(c) & \text{if $a \equiv 1 \tpmod{c}$},\\
	0 & \text{otherwise}.
	\end{cases}
\end{equation}
Given $h \in \Z$, we define the Gau{\ss} sum associated to $\chi$ by
\begin{equation}\label{Gauss-sum}
\tau(\chi, h) \coloneqq \sum_{b \tpmod c} \chi(b) e \left(\frac{bh}{c} \right).
\end{equation}
We write $\tau(\chi) \coloneqq \tau(\chi, 1)$. Multiplying~\eqref{Gauss-sum} by $\overline{\chi}(a)$ and summing over $\chi$, we derive from~\eqref{orthogonality-of-Dirichlet-characters}
\begin{equation*}
e \left(\frac{ah}{c} \right) = \frac{1}{\varphi(c)} \sum_{\chi \tpmod{c}} \overline{\chi}(a) \tau(\chi, h), \qquad (a, c) = 1.
\end{equation*}
This serves as a Fourier expansion of additive characters in terms of the multiplicative ones.

When $\chi$ is quadratic and $d$ is a positive odd squarefree integer, the Gau{\ss} sum simplifies~to
\begin{equation*}
\tau \left(\left(\frac{\cdot}{d} \right) \right) = \epsilon_{d} \sqrt{d},
\end{equation*}
where
\begin{equation*}
\epsilon_{d} = 
    \begin{cases}
    1 & \text{if $d \equiv 1 \tpmod{4}$},\\
    i & \text{if $d \equiv 3 \tpmod{4}$}.
    \end{cases}
\end{equation*}
It is straightforward to verify that the right-hand side of~\eqref{eq:quadratic-reciprocity} is equal to $\epsilon_{d} \epsilon_{n} \epsilon_{dn}^{-1}$.

\subsection{The Gamma Function}\label{the-Gamma-function}
For fixed $\sigma \in \R$, real $|\tau| \geq 3$, and any $M > 0$, we make use of Stirling's formula
\begin{equation}\label{eq:Stirling}
\Gamma(\sigma+i\tau) = e^{-\frac{\pi|\tau|}{2}} |\tau|^{\sigma-\frac{1}{2}} \exp \left(i\tau \log \frac{|\tau|}{e} \right) g_{\sigma, M}(\tau)+O_{\sigma, M}(|\tau|^{-M}),
\end{equation}
where
\begin{equation*}
g_{\sigma, M}(\tau) = \sqrt{2\pi} \exp \left(\frac{\pi}{4}(2\sigma-1)i \operatorname{sgn}(\tau) \right)+O_{\sigma, M}(|\tau|^{-1}),
\end{equation*}
and
\begin{equation*}
|\tau|^{j} g_{\sigma, M}^{(j)}(\tau) \ll_{j, \sigma, M} 1
\end{equation*}
for all fixed $j \in \N_{0}$.

\section{Automorphic Toolbox}\label{automorphic-machinery}
This section reviews the automorphic machinery to be considered in the rest of the paper. In particular, we define automorphic $L$-functions and their quadratic twists, followed by the approximate functional equation. The Vorono\u{\i} summation formula for twists is also shown.

\subsection{Automorphic Forms}
Let $\{\varphi \}$ be an orthonormal basis of Hecke--Maa{\ss} cusp forms~on the modular surface $\SL_{2}(\Z) \backslash \mathbb{H}$. We can assume without loss of generality that all $\varphi$ are real-valued. Denote by $t_{\varphi} > 1$ the spectral parameter, and by $\lambda_{\varphi}(n)$ the $n$-th Fourier coefficient. Given $t \in \R$, let $E(z, \frac{1}{2}+it)$ be the unitary Eisenstein series whose $n$-th Fourier coefficient is $\lambda(n, t) \coloneqq \sum_{ab = |n|} (\frac{a}{b})^{it}$. Let $\vartheta$ be an admissible exponent towards the Ramanujan--Petersson conjecture. At the current state of knowledge, $\vartheta \leq \frac{7}{64}$ is known; see Kim--Sarnak~\cite{Kim2003}. Nonetheless, the Ramanujan--Petersson conjecture holds \emph{on average} in the following form.
\begin{lemma}[{Rankin--Selberg bound~\cite[Lemma~1]{Iwaniec1992}}]\label{lem:Rankin-Selberg}
Keep the notation as above. Then we have for any~$\epsilon > 0$ that
\begin{equation*}
\sum_{n \leq N} |\lambda_{\varphi}(n)|^{2} \ll_{\epsilon} t_{\varphi}^{\epsilon} N.
\end{equation*}
\end{lemma}

The Fourier coefficients $\lambda_{\varphi}(n)$ also obey the Hecke multiplicativity relation
\begin{equation}\label{eq:Hecke}
\lambda_{\varphi}(mn) = \sum_{d \mid (m, n)} \mu(d) \lambda_{\varphi} \left(\frac{m}{d} \right) \lambda_{\varphi} \left(\frac{n}{d} \right), \qquad m, n \in \N.
\end{equation}

\subsection{$L$-Functions}\label{GL(2)-l-functions}
Let $\varphi$ be a Hecke--Maa{\ss} cusp form on $\SL_{2}(\Z) \backslash \mathbb{H}$ of Laplacian eigenvalue $\frac{1}{4}+t_{\varphi}^{2} \geq 0$. Let $\lambda_{\varphi}(n)$ be its $n$-th Fourier coefficient. Then the $L$-function associated to $\varphi$ is given by
\begin{equation*}
L(s, \varphi) \coloneqq \sum_{n = 1}^{\infty} \frac{\lambda_{\varphi}(n)}{n^{s}} = \prod_{p} \left(1-\frac{\lambda_{\varphi}(p)}{p^{s}}+\frac{1}{p^{2s}} \right)^{-1},
\end{equation*}
which converges absolutely for $\Re(s) > 1$, extends to the whole complex plane $\C$, and satisfies the functional equation
\begin{equation*}
\Lambda(s, \varphi) \coloneqq \pi^{-s} \Gamma \left(\frac{s+\kappa+it_{\varphi}}{2} \right) \Gamma \left(\frac{s+\kappa-it_{\varphi}}{2} \right) L(s, \varphi)
 = \epsilon(\varphi) \Lambda(1-s, \varphi),
\end{equation*}
where $\epsilon(\varphi)$ stands for the root number of modulus $1$, and
\begin{equation*}
\kappa = 
    \begin{cases}
    0 & \text{if $\epsilon(\varphi) = 1$},\\
    1 & \text{if $\epsilon(\varphi) = -1$}.
    \end{cases}
\end{equation*}
Furthermore, $L(s, \varphi) = \zeta(s)^{2}$ if $\varphi$ is Eisenstein.

\subsection{Quadratic Twists}
With the notation as above, the quadratic twist $\varphi \otimes \chi_{d}$ becomes a Hecke--Maa{\ss} newform of level $|d|^{2}$ whose $L$-function can be expressed in terms of a Dirichlet series and an Euler product, each converging absolutely for $\Re(s) > 1$:
\begin{equation*}
L(s, \varphi \otimes \chi_{d}) \coloneqq \sum_{n = 1}^{\infty} \frac{\lambda_{\varphi}(n) \chi_{d}(n)}{n^{s}} = \prod_{p} \left(1-\frac{\lambda_{\varphi}(p) \chi_{d}(p)}{p^{s}}+\frac{\chi_{d}(p)^{2}}{p^{2s}} \right)^{-1}.
\end{equation*}
It extends to the whole complex plane $\C$ and satisfies the functional equation
\begin{align*}
\Lambda(s, \varphi \otimes \chi_{d}) &\coloneqq \left(\frac{|d|}{\pi} \right)^{s} \Gamma \left(\frac{s+\kappa+it_{\varphi}}{2} \right) \Gamma \left(\frac{s+\kappa-it_{\varphi}}{2} \right) L(s, \varphi \otimes \chi_{d})\\
& = \epsilon(\varphi \otimes \chi_{d}) \Lambda(1-s, \varphi \otimes \chi_{d}),
\end{align*}
where $\epsilon(\varphi \otimes \chi_{d}) = \epsilon(\varphi) \epsilon(d)$ with $\epsilon(d) = (\frac{d}{-1}) = \pm 1$ depending on the sign of $d$. Furthermore, $L(s, \varphi \otimes \chi_{d}) = L(s, \chi_{d})^{2}$ if $\varphi$ is Eisenstein.

\subsection{The Approximate Functional Equation}\label{approximate-functional-equations}
We record a version of the approximate functional equation due to Iwaniec--Kowalski~\cite[Theorem~5.3]{IwaniecKowalski2004} applied to $L(\frac{1}{2}, \varphi \otimes \chi_{d})$.
\begin{lemma}[{Iwaniec--Kowalski~\cite[Theorem~5.3]{IwaniecKowalski2004}}]\label{Iwaniec-Kowalski}
Let $G(u)$ be any function that is even, holomorphic and bounded in the horizontal strip $-4 < \Re(u) < 4$, and normalised such that $G(0) = 1$. Then we have that
\begin{equation*}
L \left(\frac{1}{2}, \varphi \otimes \chi_{d} \right) = (1+\epsilon(\varphi \otimes \chi_{d})) \sum_{n = 1}^{\infty} \frac{\lambda_{\varphi}(n) \chi_{d}(n)}{\sqrt{n}} W \left(\frac{n}{|d|} \right)+O(|d|^{-2023}),
\end{equation*}
where for any $c > 1$,
\begin{equation}\label{eq:V}
W(y) \coloneqq \frac{1}{2\pi i} \int_{(c)} (\pi y)^{-u} G(u) \frac{\Gamma(\frac{s+u+\kappa+it_{\varphi}}{2}) \Gamma(\frac{s+u+\kappa-it_{\varphi}}{2})}{\Gamma(\frac{s+\kappa+it_{\varphi}}{2}) \Gamma(\frac{s+\kappa-it_{\varphi}}{2})} \frac{du}{u}.
\end{equation}
\end{lemma}

Note that $W(y)$ decays rapidly as $y \to \infty$ by taking $c$ suitably large in the definition~\eqref{eq:V} and then using Stirling's formula~\eqref{eq:Stirling}. Since we are only interested in positive fundamental discriminants $d$, we assume without loss of generality that $\epsilon(\varphi) = 1$, namely that $\varphi$ is even, because the central $L$-value vanishes otherwise.

\subsection{Vorono\u{\i} Summation}\label{subsec:Voronoi}
In conjunction with Poisson summation in Section~\ref{Poisson-summation}, one of the key ingredients in the proof of Theorem~\ref{main} is Vorono\u{\i} summation for Hecke--Maa{\ss} cusp forms on $\SL_{2}(\Z) \backslash \mathbb{H}$, which is thought of as applying Poisson summation (Lemma~\ref{Poisson-summation})~twice. To enable subsequent discussions, we need some notation. Let $V: (0, \infty) \to \C$ be a smooth function with compact support. Define the Hankel transform of $V$ by
\begin{equation*}
\mathring{V}_{\varphi}^{\pm}(y) \coloneqq \int_{0}^{\infty} V(x) J_{\varphi}^{\pm}(4\pi \sqrt{xy}) dx,
\end{equation*}
where
\begin{equation}\label{eq:J-phi}
J_{\varphi}^{+}(x) \coloneqq -\frac{\pi}{\cosh(\pi t_{\varphi})}(Y_{2it_{\varphi}}(x)+Y_{-2it_{\varphi}}(x)), \qquad 
J_{\varphi}^{-}(x) \coloneqq 4\epsilon(\varphi) \cosh(\pi t_{\varphi}) K_{2it_{\varphi}}(x).
\end{equation}
It is straightforward to confirm that $\mathring{V}$ is a Schwartz function (cf.~\cite{GradshteynRyzhik2007}).

We are now ready to formulate the Vorono\u{\i} summation formula; see~\cite[Proposition~2]{BlomerHarcos2012}.
\begin{lemma}[Vorono\u{\i} summation]\label{lem:Voronoi}
Let $c \in \N$ and $d \in \Z$ with $(c, d) = 1$. Let $V: (0, \infty) \to \C$ be a smooth function with compact support. Then we have for $N > 0$ that
\begin{equation*}
\sum_{n} \lambda_{\varphi}(n) e \left(\frac{dn}{c} \right) V \left(\frac{n}{N} \right)
 = \frac{N}{c} \sum_{n} \sum_{\pm} \lambda_{\varphi}(n) e \left(\mp \frac{\overline{d} n}{c} \right) \mathring{V}_{\varphi}^{\pm} \left(\frac{n}{c^{2}/N} \right).
\end{equation*}
\end{lemma}

\begin{corollary}[Vorono\u{\i} summation for twists]\label{lem:Voronoi-twists}
Let $c \in \N$ be an odd squarefree integer and $d \in \Z$ with $(c, d) = 1$. Let $V: (0, \infty) \to \C$ be a smooth function with compact support. Then we have for $N > 0$ that
\begin{equation}\label{eq:Voronoi-twists}
\sum_{n} \lambda_{\varphi}(n) \left(\frac{c}{n} \right) e \left(\frac{dn}{c} \right) V \left(\frac{n}{N} \right)
 = \frac{N}{c} \sum_{n} \sum_{\pm} \lambda_{\varphi}(n) \left(\frac{c}{n} \right) e \left(\mp \frac{\overline{d} n}{c} \right) \mathring{V}_{\varphi}^{\pm} \left(\frac{n}{c^{2}/N} \right).
\end{equation}
\end{corollary}

\begin{proof}
We use~\cite[Proposition 3.8 (iii)]{JacquetLanglands1970} or~\cite[Theorem 3.1 (ii)]{AtkinLi1978} to see that there exists a Hecke--Maa{\ss} newform $\varphi \otimes (\frac{c}{\cdot})$ of level $c^{2}$ and trivial central character such that~$\lambda_{\varphi \otimes (\frac{c}{\cdot})}(n) = \lambda_{\varphi}(n) (\frac{c}{n})$. Corollary~\ref{lem:Voronoi-twists} now follows from (a level-included version of) Lemma~\ref{lem:Voronoi}.
\end{proof}

\section{Proof of Theorem~\ref{main}}\label{proof}
In this section, we embark on the proof of Theorem~\ref{main}. Recall that our goal is to estimate
\begin{equation*}
\mathcal{M}_{\varphi}(T, H) \coloneqq \sideset{}{^{\ast}} \sum_{H \leq h \leq 2H} \ \sideset{}{^{\ast}} \sum_{T \leq n \leq 2T} 
L \left(\frac{1}{2}, \varphi \otimes \chi_{8n} \right) L \left(\frac{1}{2}, \varphi \otimes \chi_{8(n+h)} \right).
\end{equation*}
We regard $\varphi$ as fixed, and use the convention that $\epsilon$ is an arbitrarily small positive quantity, not necessarily the same in each instance. Each inequality in what follows is allowed to have an implicit constant dependent at most on $\varphi$ and $\epsilon$, unless otherwise specified.

\subsection{Trivial Bound}
Applying the Cauchy--Schwarz inequality and trivially estimating the second moment of quadratic twists imply
\begin{equation*}
\mathcal{M}_{\varphi}(T, H) \ll \sideset{}{^{\ast}} \sum_{h \ll H} \left(\sideset{}{^{\ast}} \sum_{n \ll T} \left|L \left(\frac{1}{2}, \varphi \otimes \chi_{8n} \right) \right|^{2} \right)^{\frac{1}{2}} \left(\sideset{}{^{\ast}} \sum_{n \ll T} \left|L \left(\frac{1}{2}, \varphi \otimes \chi_{8(n+h)} \right) \right|^{2} \right)^{\frac{1}{2}} \ll HT^{1+\epsilon}
\end{equation*}
for any $1 \leq H \leq T$. To establish Theorem~\ref{main}, we thus need to save roughly $HT^{-\frac{1}{4}} \geq 1$.

\subsection{Smoothing}
Upon approximating the indicator function $\mathbf{1}_{(H, 2H] \times (N, 2N]}$ by a compactly supported smooth function $W \in C_{c}^{\infty}([1, 2] \times [1, 2])$, it suffices to handle the smoothed version
\begin{equation*}
\sideset{}{^{\ast}} \sum_{h} \sideset{}{^{\ast}} \sum_{n} L \left(\frac{1}{2}, \varphi \otimes \chi_{8n} \right) L \left(\frac{1}{2}, \varphi \otimes \chi_{8(n+h)} \right) W \left(\frac{n}{T}, \frac{h}{H} \right).
\end{equation*}

\subsection{Applying the $\delta$-Symbol}
We now use the circle method (Lemma~\ref{DeltaCor}) to separate~the oscillations. Let $1 \leq C \leq \sqrt{T}$ be a parameter that we shall determine later, and fix a smooth function $U$ that takes $1$ on $[1, 2]$ and $0$ outside $[\frac{1}{2}, \frac{5}{2}]$. Then we need to analyse the expression
\begin{multline*}
\sideset{}{^{\ast}} \sum_{h} \sideset{}{^{\ast}} \sum_{n} L \left(\frac{1}{2}, \varphi \otimes \chi_{8n} \right) W \left(\frac{n}{T}, \frac{h}{H} \right) \sideset{}{^{\ast}} \sum_{m} L \left(\frac{1}{2}, \varphi \otimes \chi_{8m} \right) U \left(\frac{m}{T} \right) \delta(m = n+h)\\
 = \frac{1}{\mathcal{C}} \sum_{c} \frac{1}{c} \sideset{}{^{\ast}} \sum_{h} \sideset{}{^{\ast}} \sum_{n} L \left(\frac{1}{2}, \varphi \otimes \chi_{8n} \right) W \left(\frac{n}{T}, \frac{h}{H} \right) \sideset{}{^{\ast}} \sum_{m} L \left(\frac{1}{2}, \varphi \otimes \chi_{8m} \right) U \left(\frac{m}{T} \right)\\ 
\times \sum_{a \tpmod{c}} e \left(\frac{a(n+h-m)}{c} \right) V_{0} \left(\frac{c}{C}, \frac{n+h-m}{cC} \right)
\end{multline*}
for some $\mathcal{C} \sim C$ and a fixed smooth function $V_{0}$ satisfying $V_{0}(x, y) \ll \delta(|x|, |y| \ll 1)$. Pulling out the divisor $b = (a, c)$ in tandem with the replacement $a \mapsto -a$ yields
\begin{multline*}
\mathcal{M}_{\varphi}(T, H) \ll \frac{1}{C} \sum_{b, c} \frac{1}{bc} \ \sideset{}{^{\star}} \sum_{a \tpmod{c}} \sideset{}{^{\ast}} \sum_{h} e \left(-\frac{ah}{c} \right) \sideset{}{^{\ast}} \sum_{m} L \left(\frac{1}{2}, \varphi \otimes \chi_{8m} \right) e \left(\frac{am}{c} \right) U \left(\frac{m}{T} \right)\\
\times \sideset{}{^{\ast}} \sum_{n} L \left(\frac{1}{2}, \varphi \otimes \chi_{8n} \right) e \left(-\frac{an}{c} \right) V_{0} \left(\frac{bc}{C}, \frac{n+h-m}{bcC} \right) W \left(\frac{n}{T}, \frac{h}{H} \right)+T^{-2023}.
\end{multline*}

\subsection{Poisson Summation in $h$}
We first remove the asterisk (the squarefree condition)~on the $h$-sum via M\"{o}bius inversion, writing
\begin{multline*}
\sideset{}{^{\ast}} \sum_{h} e \left(-\frac{ah}{c} \right) W \left(\frac{n}{T}, \frac{h}{H} \right) V_{0} \left(\frac{bc}{C}, \frac{n+h-m}{bcC} \right)\\
 = \sum_{d} \mu(d) \sum_{h} e \left(-\frac{ad^{2} h}{c} \right) V_{0} \left(\frac{bc}{C}, \frac{n+d^{2} h-m}{bcC} \right) W \left(\frac{n}{T}, \frac{d^{2} h}{H} \right).
\end{multline*}
Applying Poisson summation (Lemma \ref{Fouvry-Kowalski-Michel}) to the $h$-sum shows that the right-hand side is
\begin{equation*}
H \sum_{d} \frac{\mu(d)}{d^{2}} \sum_{h} \delta(h \equiv ad^{2} \tpmod{c}) \mathcal{J}_{1}(h, m, n, c),
\end{equation*}
where\footnote{We suppress less important variables from the notation, which applies to the ensuing integral transforms.}
\begin{equation*}
\mathcal{J}_{0}(h, m, n, c) \coloneqq \int_{\R} V_{0} \left(\frac{bc}{C}, \frac{n+Hy-m}{bcC} \right) W \left(\frac{n}{T}, y \right) e \left(-\frac{hHy}{cd^{2}} \right) dy.
\end{equation*}
Repeated integration by parts ensures an arbitrary saving unless\footnote{We assume without loss of generality that $h > 0$, since the argument would be quite similar in the other case. This kind of restriction applies to the subsequent analysis.}
\begin{equation*}
h \ll \frac{cd^{2}}{H}.
\end{equation*}
The above congruence condition is solvable with $(a, c) = 1$ if and only if $(c, d^{2}) = (c, h)$. We thus factorise $c$ in terms of the Chinese Remainder Theorem (cf.~\cite[Section 6.3]{KiralYoung2021}). Write
\begin{equation}\label{eq:condition-1}
c = c_{0} c_{1}, \qquad h = h_{0} h_{1},
\end{equation}
where the factorisations may be written locally as
\begin{align*}
c_{0} &= \prod_{\nu_{p}(c) > \nu_{p}(h)} p^{\nu_{p}(c)}, \qquad c_{1} = \prod_{1 \leq \nu_{p}(c) \leq \nu_{p}(h)} p^{\nu_{p}(c)},\\
h_{0} &= \prod_{\nu_{p}(h) \geq \nu_{p}(c)} p^{\nu_{p}(h)}, \qquad h_{1} = \prod_{1 \leq \nu_{p}(h) < \nu_{p}(c)} p^{\nu_{p}(h)},
\end{align*}
with the $p$-adic valuation defined by $\nu_{p}(n) = d$ for $p^{d} \parallel n$. Alternatively, if $n^{\ast} = \prod_{p \mid n} p$, then
\begin{equation}\label{eq:condition-2}
(c_{0}, h_{0}) = 1, \qquad c_{1} \mid h_{0}, \qquad h_{1} h_{1}^{\ast} \mid c_{0}.
\end{equation}
These conditions characterise the variables $c_{0}$, $c_{1}$, $h_{0}$, $h_{1}$. Note that
\begin{equation}\label{eq:condition-3}
(c_{0}, c_{1}) = (c_{1}, h_{1}) = (h_{0}, h_{1}) = 1
\end{equation}
automatically from the other conditions. It transpires from~\eqref{eq:condition-1} and~\eqref{eq:condition-2} that $(c, h) = c_{1} h_{1}$ so that we impose the condition $c_{1} h_{1} = (c_{0} c_{1}, d^{2}) = (\frac{c_{0}}{h_{1}} c_{1} h_{1}, d^{2})$. Hence, one may decompose
\begin{equation*}
d^{2} = c_{1} d^{\prime} h_{1},
\end{equation*}
where the new variable $d^{\prime}$ is only subject to the restriction $(\frac{c_{0}}{h_{1}}, d^{\prime}) = 1$, namely $(c_{0}, d^{\prime}) = 1$, since $\frac{c_{0}}{h_{1}}$ shares the same prime factors as $c_{0}$. It is now possible to recast the initial congruence condition $h \equiv ad^{2} \tpmod{c}$ as
\begin{equation*}
h_{0} h_{1} \equiv ac_{1} d^{\prime} h_{1} \tpmod{c_{0} c_{1}} \qquad \Longleftrightarrow \qquad a \equiv \frac{h_{0} \overline{d^{\prime}}}{c_{1}} \tpmod{\frac{c_{0}}{h_{1}}},
\end{equation*}
where $\overline{d^{\prime}}$ is taken to be the multiplicative inverse modulo $c_{0}$ thanks to~\eqref{eq:condition-2}. This condition can further be rewritten as
\begin{equation*}
a \equiv \frac{h_{0} \overline{d^{\prime}}}{c_{1}}+\frac{c_{0} u}{h_{1}} \tpmod{c_{0}}, \qquad u \tpmod{h_{1}},
\end{equation*}
where $u$ runs through all residue classes modulo $h_{1}$, since as soon as $a$ is coprime to $\frac{c_{0}}{h_{1}}$, it is also coprime to $c_{0}$. The Chinese Remainder Theorem implies that the sum over $a$ equals
\begin{equation*}
h_{1} e \left(\frac{\overline{c_{1} d^{\prime}} h_{0}(m-n)}{c_{0} c_{1}} \right) S(m-n; 0; c_{1}) \delta(m \equiv n \tpmod{h_{1}}).
\end{equation*}
It is advantageous to open the Kloosterman sum (or the Ramanujan sum) in the form
\begin{equation*}
\sum_{c_{2} \mid (c_{1}, m-n)} \mu \left(\frac{c_{1}}{c_{2}} \right) c_{2},
\end{equation*}
and replace $c_{1} \mapsto c_{1} c_{2}$.

As a result, we are led to the expression
\begin{multline*}
\mathcal{M}_{\varphi}(T, H) \ll \frac{H}{C} \sum_{\substack{b, c_{0}, c_{1}, c_{2}, d \\ d^{2} = c_{1} c_{2} d^{\prime} h_{1} \\ (c_{0}, d^{\prime}) = 1}} \frac{\mu(c_{1}) \mu(d)}{bc_{0} c_{1}^{2} c_{2} d^{\prime}} \sum_{h_{0} \ll \frac{c_{0} c_{1}^{2} c_{2}^{2} d^{\prime}}{H}} \ \underset{m \equiv n \tpmod{c_{2} h_{1}}}{\sideset{}{^{\ast}} \sum \sideset{}{^{\ast}} \sum} L \left(\frac{1}{2}, \varphi \otimes \chi_{8m} \right) L \left(\frac{1}{2}, \varphi \otimes \chi_{8n} \right)\\
\times e \left(\frac{\overline{c_{1} c_{2} d^{\prime}} h_{0}(m-n)}{c_{0} c_{1}} \right) U \left(\frac{m}{T} \right) \mathcal{J}_{0}(h_{0} h_{1}, m, n, c_{0} c_{1} c_{2})+T^{-2023},
\end{multline*}
where we drop the conditions~\eqref{eq:condition-1} and~\eqref{eq:condition-2} from summations for simplicity.

\subsection{Poisson Summation in $m$}
It now follows from the approximate functional equation (Lemma~\ref{approximate-functional-equations}) that there exists a smooth function $W_{1} \in C_{c}^{\infty}(\R)$ supported on $[\frac{1}{2}, \frac{5}{2}]$ such that
\begin{multline*}
\sideset{}{^{\ast}} \sum_{m \equiv n \tpmod{c_{2} h_{1}}} L \left(\frac{1}{2}, \varphi \otimes \chi_{8m} \right) e \left(\frac{\overline{c_{1} c_{2} d^{\prime}} h_{0} m}{c_{0} c_{1}} \right) U \left(\frac{m}{T} \right) \mathcal{J}_{0}(h_{0} h_{1}, m, n, c_{0} c_{1} c_{2})\\
 = 2 \sum_{\ell_{1} = 1}^{\infty} \frac{\lambda_{\varphi}(\ell_{1})}{\sqrt{\ell_{1}}} \ \sideset{}{^{\ast}} \sum_{m \equiv n \tpmod{c_{2} h_{1}}} \left(\frac{8m}{\ell_{1}} \right) e \left(\frac{\overline{c_{1} c_{2} d^{\prime}} h_{0} m}{c_{0} c_{1}} \right) U \left(\frac{m}{T} \right) W_{1} \left(\frac{\ell_{1}}{8m} \right)\\
\times \mathcal{J}_{0}(h_{0} h_{1}, m, n, c_{0} c_{1} c_{2})+T^{-2023}.
\end{multline*}
We remove the asterisk and execute Poisson summation (Lemma~\ref{Fouvry-Kowalski-Michel}) in the $m$-sum, deducing
\begin{multline*}
\frac{2T}{c_{0} c_{1} c_{2} h_{1}} \sum_{\ell_{1} = 1}^{\infty} \frac{\lambda_{\varphi}(\ell_{1})}{\ell_{1}^{\frac{3}{2}}} \sum_{(e, \ell_{1}) = 1} \frac{\mu(e)}{e^{2}} \sum_{m} \sum_{\alpha \tpmod{c_{0} c_{1} c_{2} h_{1} \ell_{1}}} \left(\frac{8\alpha}{\ell_{1}} \right)\\
\times e \left(\frac{\alpha \overline{c_{1} c_{2} d^{\prime}} e^{2} h_{0}}{c_{0} c_{1}}+\frac{\alpha m}{c_{0} c_{1} c_{2} h_{1} \ell_{1}} \right) \delta(\alpha e^{2} \equiv n \tpmod{c_{2} h_{1}}) \mathcal{J}_{1}(h_{0} h_{1}, \ell_{1}, m, n, c_{0} c_{1} c_{2}),
\end{multline*}
where
\begin{equation*}
\mathcal{J}_{1}(h_{0} h_{1}, \ell_{1}, m, n, c_{0} c_{1} c_{2}) \coloneqq \int_{\R} U(y) W_{1} \left(\frac{\ell_{1}}{8Ty} \right) \mathcal{J}_{0}(h_{0} h_{1}, Ty, n, c_{0} c_{1} c_{2}) e \left(-\frac{mTy}{c_{0} c_{1} c_{2} e^{2} h_{1} \ell_{1}} \right) dy.
\end{equation*}
Repeated integration by parts ensures an arbitrary saving unless
\begin{equation*}
m \ll \frac{c_{0} c_{1} c_{2} e^{2} h_{1} \ell_{1}}{T}.
\end{equation*}

\subsection{Poisson Summation in $n$}
We execute Poisson summation (Lemma~\ref{Fouvry-Kowalski-Michel}) in the $n$-sum in the same manner as above, deducing
\begin{align*}
\mathcal{M}_{\varphi}(T, H) &\ll \frac{HT^{2}}{C} \sum_{\substack{b, c_{0}, c_{1}, c_{2}, d, e, f \\ d^{2} = c_{1} c_{2} d^{\prime} h_{1} \\ (c_{0}, d^{\prime}) = 1}} \frac{\mu(c_{1}) \mu(d) \mu(e) \mu(f)}{bc_{0}^{3} c_{1}^{4} c_{2}^{3} d^{\prime} e^{2} f^{2} h_{1}^{2}} \sum_{h_{0} \ll \frac{c_{0} c_{1}^{2} c_{2}^{2} d^{\prime}}{H}} \sum_{\substack{T^{1-\epsilon} \ll \ell_{1}, \ell_{2} \ll T^{1+\epsilon} \\ (\ell_{1}, e) = (\ell_{2}, f) = 1}} \frac{\lambda_{\varphi}(\ell_{1}) \lambda_{\varphi}(\ell_{2})}{(\ell_{1} \ell_{2})^{\frac{3}{2}}}\\
& \quad \times \sum_{m \ll \frac{c_{0} c_{1} c_{2} e^{2} h_{1} \ell_{1}}{T}} \sum_{n \ll \frac{c_{0} c_{1} c_{2} f^{2} h_{1} \ell_{2}}{T}} \sum_{\substack{\alpha \tpmod{c_{0} c_{1} c_{2} h_{1} \ell_{1}} \\ \beta \tpmod{c_{0} c_{1} c_{2} h_{1} \ell_{2}}}} \left(\frac{8\alpha}{\ell_{1}} \right) \left(\frac{8\beta}{\ell_{2}} \right)\\
& \quad \times e \left(\frac{\alpha \overline{c_{1} c_{2} d^{\prime}} e^{2} h_{0}}{c_{0} c_{1}}-\frac{\beta \overline{c_{1} c_{2} d^{\prime}} f^{2} h_{0}}{c_{0} c_{1}}+\frac{\alpha m}{c_{0} c_{1} c_{2} h_{1} \ell_{1}}+\frac{\beta n}{c_{0} c_{1} c_{2} h_{1} \ell_{2}} \right)\\
&\qquad \times \delta(\alpha e^{2} \equiv \beta f^{2} \tpmod{c_{2} h_{1}}) \mathcal{J}_{2}(h_{0} h_{1}, \ell_{1}, \ell_{2}, m, n, c_{0} c_{1} c_{2})+T^{-2023},
\end{align*}
where
\begin{equation*}
\mathcal{J}_{2}(h_{0} h_{1}, \ell_{1}, \ell_{2}, m, n, c_{0} c_{1} c_{2}) \coloneqq \int_{\R} W_{2} \left(\frac{\ell_{2}}{8Ty} \right) \mathcal{J}_{1}(h_{0} h_{1}, \ell_{1}, m, Ty, c_{0} c_{1} c_{2}) e \left(-\frac{nTy}{c_{0} c_{1} c_{2} f^{2} h_{1} \ell_{2}} \right) dy
\end{equation*}
for a smooth function $W_{2} \in C_{c}^{\infty}(\R)$ supported on $[\frac{1}{2}, \frac{5}{2}]$.

\subsection{A Simplification of Character Sums}
In anticipation of the forthcoming analysis, it is convenient to simplify the character sums appearing in the above section. Detecting the restriction $\alpha e^{2} \equiv \beta f^{2} \tpmod{c_{2} h_{1}}$ via additive characters modulo $c_{2} h_{1}$, we derive
\begin{multline*}
\frac{1}{c_{2} h_{1}} \sum_{\substack{\alpha \tpmod{c_{0} c_{1} c_{2} h_{1} \ell_{1}} \\ \beta \tpmod{c_{0} c_{1} c_{2} h_{1} \ell_{2}} \\ \gamma \tpmod{c_{2} h_{1}}}} \left(\frac{8\alpha}{\ell_{1}} \right) \left(\frac{8\beta}{\ell_{2}} \right) e \left(\frac{\alpha \overline{c_{1} c_{2} d^{\prime}} e^{2} h_{0}}{c_{0} c_{1}}-\frac{\beta \overline{c_{1} c_{2} d^{\prime}} f^{2} h_{0}}{c_{0} c_{1}} \right)\\
\times e \left(\frac{\alpha m}{c_{0} c_{1} c_{2} h_{1} \ell_{1}}+\frac{\beta n}{c_{0} c_{1} c_{2} h_{1} \ell_{2}}+\frac{\gamma(\alpha e^{2}-\beta f^{2})}{c_{2} h_{1}} \right).
\end{multline*}
The sum over $\alpha$ vanishes unless
\begin{equation}\label{eq:alpha}
\overline{c_{1}} c_{2} \overline{c_{2} d^{\prime}} e^{2} h_{0} h_{1} \ell_{1}+m+\gamma c_{0} c_{1} e^{2} \ell_{1} \equiv 0 \tpmod{c_{0} c_{1} c_{2} h_{1}},
\end{equation}
in which case it is
\begin{equation}\label{eq:alpha-evaluation}
c_{0} c_{1} c_{2} h_{1} \left(\frac{8c_{0} c_{1} c_{2} h_{1} m}{\ell_{1}} \right) \tau \left(\left(\frac{\cdot}{\ell_{1}} \right) \right).
\end{equation}
Similarly, the sum over $\beta$ vanishes unless
\begin{equation}\label{eq:beta}
-\overline{c_{1}} c_{2} \overline{c_{2} d^{\prime}} f^{2} h_{0} h_{1} \ell_{2}+n-\gamma c_{0} c_{1} f^{2} \ell_{2} \equiv 0 \tpmod{c_{0} c_{1} c_{2} h_{1}},
\end{equation}
in which case it is
\begin{equation}\label{eq:beta-evaluation}
c_{0} c_{1} c_{2} h_{1} \left(\frac{8c_{0} c_{1} c_{2} h_{1} n}{\ell_{2}} \right) \tau \left(\left(\frac{\cdot}{\ell_{2}} \right) \right).
\end{equation}
Furthermore, we factorise the sum over $\gamma$ into sums over $\gamma_{1} \tpmod{c_{2}}$ and $\gamma_{2} \tpmod{h_{1}}$.~The combination of~\eqref{eq:alpha} and~\eqref{eq:beta} shows
\begin{gather*}
\overline{c_{1} d^{\prime}} e^{2} h_{0} h_{1} \ell_{1}+m \equiv -\overline{c_{1} d^{\prime}} f^{2} h_{0} h_{1} \ell_{2}+n \equiv 0 \tpmod{c_{0}}, \qquad m \equiv n \equiv 0 \tpmod{c_{1}},\\
m+\gamma_{1} c_{0} c_{1} e^{2} \ell_{1} \equiv n-\gamma_{1} c_{0} c_{1} f^{2} \ell_{2} \equiv 0 \tpmod{c_{2}}, \qquad m \equiv n \equiv 0 \tpmod{h_{1}}.
\end{gather*}
Note that~\eqref{eq:condition-2} and~\eqref{eq:condition-3} imply in particular that $(c_{0} c_{1}, c_{2} h_{1}) = (c_{1}, c_{2})h_{1}$ and $(c_{2}, h_{1}) = 1$. By an elementary consideration, the congruence condition modulo $c_{0}$ boils down to
\begin{equation*}
f^{2} \overline{\ell_{1}} m \equiv -e^{2} \overline{\ell_{2}} n \tpmod{c_{0}},
\end{equation*}
where we assume $(c_{0}, \ell_{1} \ell_{2}) = 1$ due to the presence of quadratic characters in~\eqref{eq:alpha-evaluation} and~\eqref{eq:beta-evaluation}. Similarly, the congruence condition modulo $c_{2}$ boils down to
\begin{equation*}
f^{2} \overline{\ell_{1}} m \equiv -e^{2} \overline{\ell_{2}} n \tpmod{c_{2}},
\end{equation*}
which altogether does not depend on $\gamma_{1} \tpmod{c_{2}}$.

Gathering the above computations together leads to
\begin{align*}
\mathcal{M}_{\varphi}(T, H) &\ll T^{2} \sum_{\substack{b, c_{1}, c_{2}, d, e, f \\ d^{2} = c_{1} c_{2} d^{\prime} h_{1}}} \frac{\mu(c_{1}) \mu(d) \mu(e) \mu(f)}{b^{2} c_{1} e^{2} f^{2}} \sup_{\substack{c_{0} \ll \frac{C}{bc_{1} c_{2}} \\ (c_{0}, d^{\prime}) = 1}} \sup_{h_{0} \ll \frac{c_{1} c_{2} d^{\prime} C}{bH}} \sum_{\substack{T^{1-\epsilon} \ll \ell_{1}, \ell_{2} \ll T^{1+\epsilon} \\ (\ell_{1}, c_{1} eh_{1}) = (\ell_{2}, c_{1} fh_{1}) = 1}} \frac{\lambda_{\varphi}(\ell_{1}) \lambda_{\varphi}(\ell_{2})}{(\ell_{1} \ell_{2})^{\frac{3}{2}}}\\
& \quad \times \sum_{m \ll \frac{e^{2} \ell_{1} C}{bc_{1} T}} \sum_{n \ll \frac{f^{2} \ell_{2} C}{bc_{1} T}} \left(\frac{8c_{0} c_{2} m}{\ell_{1}} \right) \left(\frac{8c_{0} c_{2} n}{\ell_{2}} \right) \tau \left(\left(\frac{\cdot}{\ell_{1}} \right) \right) \tau \left(\left(\frac{\cdot}{\ell_{2}} \right) \right)\\
& \quad \times \delta(e^{2} \ell_{1} n \equiv -f^{2} \ell_{2} m \tpmod{\frac{c_{0} c_{2}}{\delta}}) \mathcal{J}_{2}(h_{0} h_{1}, \ell_{1}, \ell_{2}, c_{1} h_{1} m, c_{1} h_{1} n, c_{0} c_{1} c_{2})+T^{-2023},
\end{align*}
where $\delta = (c_{1} h_{1}, c_{0} c_{2}) = (c_{1}, c_{2})h_{1}$. It is convenient to restrict our attention to odd squarefree integers $\ell_{1}$ and $\ell_{2}$ so that the Gau{\ss} sums simplify to
\begin{equation*}
\tau \left(\left(\frac{\cdot}{\ell_{1}} \right) \right) = \epsilon_{\ell_{1}} \sqrt{\ell_{1}}, \qquad 
\tau \left(\left(\frac{\cdot}{\ell_{2}} \right) \right) = \epsilon_{\ell_{2}} \sqrt{\ell_{2}}.
\end{equation*}
Without loss of generality, we shall focus on the case where $\epsilon_{\ell_{1}} = \epsilon_{\ell_{2}} = 1$.

\subsection{Amplification}
We introduce an amplification parameter $L \geq 1$ at our disposal. Then
\begin{align*}
\mathcal{M}_{\varphi}(T, H) &\ll T^{\epsilon} \sum_{\substack{b, c_{1}, c_{2}, d, e, f \\ d^{2} = c_{1} c_{2} d^{\prime} h_{1}}} \frac{\mu(c_{1}) \mu(d) \mu(e) \mu(f)}{b^{2} c_{1} e^{2} f^{2}} \sup_{\substack{c_{0} \ll \frac{CL}{bc_{1} c_{2}} \\ (c_{0}, d^{\prime}) = 1}} \sup_{h_{0} \ll \frac{c_{1} c_{2} d^{\prime} C}{bH}} \ \sideset{}{^{\ast}} \sum_{\substack{T^{1-\epsilon} \ll \ell_{1}, \ell_{2} \ll T^{1+\epsilon} \\ (\ell_{1}, c_{1} eh_{1}) = (\ell_{2}, c_{1} fh_{1}) = 1}}\\
& \quad \times \lambda_{\varphi}(\ell_{1}) \lambda_{\varphi}(\ell_{2}) \sum_{m \ll \frac{e^{2} \ell_{1} C}{bc_{1} T}} \sum_{n \ll \frac{f^{2} \ell_{2} C}{bc_{1} T}} \left(\frac{8c_{0} c_{2} m}{\ell_{1}} \right) \left(\frac{8c_{0} c_{2} n}{\ell_{2}} \right)\\
& \quad \times \delta(e^{2} \ell_{1} n \equiv -f^{2} \ell_{2} m \tpmod{\frac{c_{0} c_{2}}{\delta}}) \mathcal{J}_{2}(h_{0} h_{1}, \ell_{1}, \ell_{2}, c_{1} h_{1} m, c_{1} h_{1} n, c_{0} c_{1} c_{2})+T^{-2023},
\end{align*}
where we pull out the factor $(\ell_{1} \ell_{2})^{-1}$ by partial summation. The determination of $L$ dictates the quality of the final bound.

\subsection{Divisor Switching}
We now perform divisor switching and write
\begin{equation*}
f^{2} \ell_{2} m+e^{2} \ell_{1} n = \frac{c_{0} c_{2}}{\delta} \cdot q, \qquad q \ll \frac{\delta e^{2} f^{2} \ell_{1} \ell_{2}}{LT}.
\end{equation*}
It follows from quadratic reciprocity~\eqref{eq:quadratic-reciprocity} and the assumption $\epsilon_{\ell_{1}} = \epsilon_{\ell_{2}} = 1$ that
\begin{equation*}
\left(\frac{8c_{0} c_{2} m}{\ell_{1}} \right) \left(\frac{8c_{0} c_{2} n}{\ell_{2}} \right)
 = \left(\frac{8\delta q}{\ell_{1} \ell_{2}} \right) \delta((\ell_{1}, fm) = (\ell_{2}, en) = (\ell_{1}, \ell_{2}) = 1).
\end{equation*}
Therefore, we are led to the expression
\begin{align*}
\mathcal{M}_{\varphi}(T, H) &\ll T^{\epsilon} \sum_{\substack{b, c_{1}, c_{2}, d, e, f \\ d^{2} = c_{1} c_{2} d^{\prime} h_{1}}} \frac{\mu(c_{1}) \mu(d) \mu(e) \mu(f)}{b^{2} c_{1} e^{2} f^{2}} \sup_{h_{0} \ll \frac{c_{1} c_{2} d^{\prime} C}{bH}} \ \sideset{}{^{\ast}} \sum_{\substack{T^{1-\epsilon} \ll \ell_{1}, \ell_{2} \ll T^{1+\epsilon} \\ (\ell_{1} \ell_{2}, c_{1} efh_{1}) = (\ell_{1}, \ell_{2}) = 1}} \lambda_{\varphi}(\ell_{1}) \lambda_{\varphi}(\ell_{2})\\
& \quad \times \sum_{\substack{q \ll \frac{\delta e^{2} f^{2} \ell_{1} \ell_{2}}{LT} \\ (q, ef) = 1}} \sum_{\substack{m \ll \frac{e^{2} \ell_{1} C}{bc_{1} T} \\ (m, q) = 1}} \sum_{\substack{n \ll \frac{f^{2} \ell_{2} C}{bc_{1} T} \\ (n, q) = 1}} \left(\frac{8\delta q}{\ell_{1} \ell_{2}} \right) \delta(e^{2} \ell_{1} n \equiv -f^{2} \ell_{2} m \tpmod{q})\\
& \quad \times \mathcal{J}_{2}(h_{0} h_{1}, \ell_{1}, \ell_{2}, c_{1} h_{1} m, c_{1} h_{1} n, c_{0} c_{1} c_{2})+T^{-2023},
\end{align*}
where we attach the restriction $(q, ef) = 1$ for technical brevity.

\subsection{Poisson Summation in $m$}
Using Poisson summation (Lemma~\ref{Fouvry-Kowalski-Michel}) in the $m$-sum yields
\begin{multline*}
\sum_{\substack{m \ll \frac{e^{2} \ell_{1} C}{bc_{1} T} \\ (m, q) = 1}} \delta(e^{2} \ell_{1} n \equiv -f^{2} \ell_{2} m \tpmod{q}) \mathcal{J}_{2}(h_{0} h_{1}, \ell_{1}, \ell_{2}, c_{1} h_{1} m, c_{1} h_{1} n, c_{0} c_{1} c_{2})\\
 = \frac{e^{2} \ell_{1} C}{bc_{1} qT} \sum_{m} e \left(-\frac{e^{2} \overline{f}^{2} \ell_{1} \overline{\ell_{2}} mn}{q} \right) \mathcal{J}_{3}(h_{0}, h_{1}, \ell_{1}, \ell_{2}, m, n, c_{1}, c_{2}),
\end{multline*}
where
\begin{equation*}
\mathcal{J}_{3}(h_{0}, h_{1}, \ell_{1}, \ell_{2}, m, n, c_{1}, c_{2}) \coloneqq \int_{\R} \mathcal{J}_{2} \left(h_{0} h_{1}, \ell_{1}, \ell_{2}, \frac{e^{2} h_{1} \ell_{1} Cy}{bT}, c_{1} h_{1} n, c_{0} c_{1} c_{2} \right) e \left(-\frac{e^{2} \ell_{1} Cmy}{bc_{1} qT} \right) dy.
\end{equation*}
Repeated integration by parts ensures an arbitrary saving unless
\begin{equation*}
m \ll \frac{bc_{1} qT}{e^{2} \ell_{1} C}.
\end{equation*}

\subsection{Poisson Summation in $n$}
Using Poisson summation (Lemma~\ref{Fouvry-Kowalski-Michel}) in the $n$-sum yields
\begin{multline*}
\sum_{\substack{n \ll \frac{f^{2} \ell_{2} C}{bc_{1} T} \\ (n, q) = 1}} e \left(-\frac{e^{2} \overline{f}^{2} \ell_{1} \overline{\ell_{2}} mn}{q} \right) \mathcal{J}_{3}(h_{0}, h_{1}, \ell_{1}, \ell_{2}, m, n, c_{1}, c_{2})\\
 = \frac{f^{2} \ell_{2} C}{bc_{1} qT} \sum_{n} \delta(e^{2} \ell_{1} m \equiv f^{2} \ell_{2} n \tpmod{q}) \mathcal{J}_{4}(h_{0}, h_{1}, \ell_{1}, \ell_{2}, m, n, c_{1}, c_{2}),
\end{multline*}
where
\begin{equation*}
\mathcal{J}_{4}(h_{0}, h_{1}, \ell_{1}, \ell_{2}, m, n, c_{1}, c_{2}) \coloneqq \int_{\R} \mathcal{J}_{3} \left(h_{0}, h_{1}, \ell_{1}, \ell_{2}, m, \frac{f^{2} \ell_{2} Cy}{bc_{1} T}, c_{1}, c_{2} \right) e \left(-\frac{f^{2} \ell_{2} Cny}{bc_{1} qT} \right) dy.
\end{equation*}
Repeated integration by parts ensures an arbitrary saving unless
\begin{equation*}
n \ll \frac{bc_{1} qT}{f^{2} \ell_{2} C}.
\end{equation*}

Altogether, we obtain
\begin{align*}
\mathcal{M}_{\varphi}(T, H) &\ll C^{2} T^{\epsilon} \sum_{\substack{b, c_{1}, c_{2}, d, e, f \\ d^{2} = c_{1} c_{2} d^{\prime} h_{1}}} \frac{\mu(c_{1}) \mu(d) \mu(e) \mu(f)}{b^{4} c_{1}^{3}} \sup_{h_{0} \ll \frac{c_{1} c_{2} d^{\prime} C}{bH}} \ \sideset{}{^{\ast}} \sum_{\substack{T^{1-\epsilon} \ll \ell_{1}, \ell_{2} \ll T^{1+\epsilon} \\ (\ell_{1} \ell_{2}, c_{1} efh_{1}) = (\ell_{1}, \ell_{2}) = 1}} \lambda_{\varphi}(\ell_{1}) \lambda_{\varphi}(\ell_{2})\\
& \quad \times \sum_{\substack{q \ll \frac{\delta e^{2} f^{2} T^{1+\epsilon}}{L} \\ (q, ef) = 1}} \frac{1}{q^{2}} \sum_{\substack{m \ll \frac{bc_{1} qT^{\epsilon}}{e^{2} C} \\ (m, q) = 1}} \sum_{\substack{n \ll \frac{bc_{1} qT^{\epsilon}}{f^{2} C} \\ (n, q) = 1}} \left(\frac{8\delta q}{\ell_{1} \ell_{2}} \right) \delta(e^{2} \ell_{1} m \equiv f^{2} \ell_{2} n \tpmod{q})\\
& \quad \times \mathcal{J}_{4}(h_{0}, h_{1}, \ell_{1}, \ell_{2}, m, n, c_{1}, c_{2})+T^{-2023}.
\end{align*}

\subsection{Orthogonality}
Detecting the congruence condition at hand via primitive additive characters modulo $q$, we derive
\begin{equation*}
\delta(e^{2} \ell_{1} m \equiv f^{2} \ell_{2} n \tpmod{q}) = \sum_{q_{1} \mid q} \ \sideset{}{^{\star}} \sum_{a \tpmod{q_{1}}} e \left(\frac{a(e^{2} \ell_{1} m-f^{2} \ell_{2} n)}{q_{1}} \right).
\end{equation*}
This is thought of as an orthogonality relation in terms of the Ramanujan sum.

\subsection{Vorono\u{\i} Summation in $\ell_{1}$}
To circumvent increase of more variables, we remove~the asterisk and the coprimality conditions on the $\ell_{1}$-sum. For general non-squarefree $\ell_{1}$, we may utilise M\"{o}bius inversion and the Hecke multiplicativity relation~\eqref{eq:Hecke} to estimate peripheral variables trivially. Applying Vorono\u{\i} summation (Lemma~\ref{lem:Voronoi-twists}) in the $\ell_{1}$-sum implies
\begin{multline*}
\sum_{T^{1-\epsilon} \ll \ell_{1} \ll T^{1+\epsilon}} \lambda_{\varphi \otimes (\frac{q}{\cdot})}(\ell_{1}) e \left(\frac{ae^{2} \ell_{1} m}{q_{1}} \right) \mathcal{J}_{4}(h_{0}, h_{1}, \ell_{1}, \ell_{2}, m, n, c_{1}, c_{2})\\
 = \frac{T}{q} \sum_{\ell_{1}} \sum_{\sigma \in \{\pm \}} \lambda_{\varphi \otimes (\frac{q}{\cdot})}(\ell_{1}) e \left(\sigma \frac{\overline{ae^{2}} \ell_{1} \overline{m}}{q_{1}} \right) \mathcal{J}_{5}(h_{0}, h_{1}, \ell_{1}, \ell_{2}, m, n, c_{1}, c_{2}),
\end{multline*}
where
\begin{equation*}
\mathcal{J}_{5}^{\sigma}(h_{0}, h_{1}, \ell_{1}, \ell_{2}, m, n, c_{1}, c_{2}) \coloneqq \int_{0}^{\infty} \mathcal{J}_{4}(h_{0}, h_{1}, x, \ell_{2}, m, n, c_{1}, c_{2}) J_{\varphi}^{-\sigma} \left(\frac{4\pi \sqrt{\ell_{1} Tx}}{q} \right) dx
\end{equation*}
with $J_{\varphi}^{\pm}(x)$ defined in~\eqref{eq:J-phi}. Repeated integration by parts ensures an arbitrary saving unless
\begin{equation*}
\ell_{1} \ll \frac{q^{2+\epsilon}}{T}.
\end{equation*}

\subsection{Vorono\u{\i} Summation in $\ell_{2}$}
In a similar fashion, we make use of Vorono\u{\i} summation (Lemma~\ref{lem:Voronoi-twists}) in the $\ell_{2}$-sum, deducing
\begin{multline*}
\sum_{T^{1-\epsilon} \ll \ell_{2} \ll T^{1+\epsilon}} \lambda_{\varphi \otimes (\frac{q}{\cdot})}(\ell_{2}) e \left(-\frac{af^{2} \ell_{2} n}{q_{1}} \right) \mathcal{J}_{5}^{\sigma}(h_{0}, h_{1}, \ell_{1}, \ell_{2}, m, n, c_{1}, c_{2})\\
 = \frac{T}{q} \sum_{\ell_{1}} \sum_{\tau \in \{\pm \}} \lambda_{\varphi \otimes (\frac{q}{\cdot})}(\ell_{1}) e \left(-\tau \frac{\overline{af^{2}} \ell_{2} \overline{n}}{q_{1}} \right) \mathcal{J}_{6}^{\sigma, \tau}(h_{0}, h_{1}, \ell_{1}, \ell_{2}, m, n, c_{1}, c_{2}),
\end{multline*}
where
\begin{equation*}
\mathcal{J}_{6}^{\sigma, \tau}(h_{0}, h_{1}, \ell_{1}, \ell_{2}, m, n, c_{1}, c_{2}) \coloneqq \int_{0}^{\infty} \mathcal{J}_{5}^{\sigma}(h_{0}, h_{1}, \ell_{1}, x, m, n, c_{1}, c_{2}) J_{\varphi}^{-\tau} \left(\frac{4\pi \sqrt{\ell_{2} Tx}}{q} \right) dx.
\end{equation*}
Repeated integration by parts ensures an arbitrary saving unless
\begin{equation*}
\ell_{2} \ll \frac{q^{2+\epsilon}}{T}.
\end{equation*}

\subsection{Endgame}
Assembling the observations in the antecedent sections, we arrive at
\begin{align*}
\mathcal{M}_{\varphi}(T, H) &\ll C^{2} T^{2+\epsilon} \sum_{\substack{b, c_{1}, c_{2}, d, e, f \\ d^{2} = c_{1} c_{2} d^{\prime} h_{1}}} \frac{\mu(c_{1}) \mu(d) \mu(e) \mu(f)}{b^{4} c_{1}^{3}} \sup_{h_{0} \ll \frac{c_{1} c_{2} d^{\prime} C}{bH}} \sum_{\substack{q \ll \frac{\delta e^{2} f^{2} T^{1+\epsilon}}{L} \\ (q, ef) = 1}} \frac{1}{q^{4}}\sum_{\ell_{1}, \ell_{2} \ll \frac{q^{2+\epsilon}}{T}}\\
& \quad \times \lambda_{\varphi}(\ell_{1}) \lambda_{\varphi}(\ell_{2}) \sum_{\sigma, \tau \in \{\pm \}} \sum_{\substack{m \ll \frac{bc_{1} qT^{\epsilon}}{e^{2} C} \\ (m, q) = 1}} \sum_{\substack{n \ll \frac{bc_{1} qT^{\epsilon}}{f^{2} C} \\ (n, q) = 1}} \left(\frac{8\delta q}{\ell_{1} \ell_{2}} \right) \delta(\sigma e^{2} \ell_{2} m \equiv \tau f^{2} \ell_{1} n \tpmod{q})\\
& \quad \times \mathcal{J}_{6}^{\sigma, \tau}(h_{0}, h_{1}, \ell_{1}, \ell_{2}, m, n, c_{1}, c_{2})+T^{-2023}.
\end{align*}
where we execute the sum over $a$ via orthogonality after replacing $a \mapsto \overline{a}$. Upon applying the Rankin--Selberg~bound (Lemma~\ref{lem:Rankin-Selberg}) and estimating everything trivially, it transpires that
\begin{align*}
\mathcal{M}_{\varphi}(T, H) &\ll C^{2} T^{2+\epsilon} \sum_{\substack{b, c_{1}, c_{2}, d, e, f \\ d^{2} = c_{1} c_{2} d^{\prime} h_{1}}} \frac{\mu(c_{1}) \mu(d) \mu(e) \mu(f)}{b^{4} c_{1}^{3}} \sum_{q \ll \frac{\delta e^{2} f^{2} T^{1+\epsilon}}{L}} \frac{1}{q^{4}} \frac{q^{2+\epsilon}}{T} \frac{q^{2+\epsilon}}{T} \frac{bc_{1} qT^{\epsilon}}{e^{2} C} \frac{bc_{1} qT^{\epsilon}}{f^{2} C} \frac{1}{\sqrt{q}}\\
&\ll T^{\frac{5}{2}+\epsilon} \sum_{\substack{c_{1}, c_{2}, d, e, f \\ d^{2} = c_{1} c_{2} d^{\prime} h_{1}}} \mu(c_{1}) \mu(d) c_{1}^{-1} \delta^{\frac{5}{2}} e^{3} f^{3} L^{-\frac{5}{2}}.
\end{align*}
Theorem~\ref{main} then follows from the optimisations
\begin{equation*}
C = \sqrt{T}, \qquad L = c_{1} c_{2} \delta e^{2} f^{2} \sqrt{T}.
\end{equation*}
The proof is complete.\qed


\providecommand{\bysame}{\leavevmode\hbox to3em{\hrulefill}\thinspace}
\providecommand{\MR}{\relax\ifhmode\unskip\space\fi MR }
\providecommand{\MRhref}[2]{%
  \href{http://www.ams.org/mathscinet-getitem?mr=#1}{#2}
}
\providecommand{\href}[2]{#2}

\end{document}